\documentclass[12pt,a4paper]{amsart}
\usepackage[utf8]{inputenc}

\usepackage{graphicx,color}
\usepackage{amsmath}
\usepackage{amssymb}
\usepackage{amsfonts}
\usepackage{amsthm}
\usepackage{amscd}
\usepackage{ae}
\usepackage[T1]{fontenc}

\newcommand{\D}{\mathbb{D}}

\font\sets=msbm10 scaled \magstep1
\def\R{\text{\sets R}}

\def\C{\text{\sets C}}

\newcommand{\dist}{\mathrm{dist}}

\renewcommand\Re{\operatorname{Re}}
\renewcommand\Im{\operatorname{Im}}
\renewcommand\arg{\operatorname{arg}}

\renewcommand{\bar}[1]{\overline{#1}}

\newcommand{\diam}{{\rm diam \,}}

\theoremstyle{plain}
\newtheorem{theorem}{Theorem}

\newtheorem{lemma}[theorem]{Lemma}
\newtheorem{corollary}[theorem]{Corollary}

\theoremstyle{definition}
\newtheorem{remark}[theorem]{Remark}

\newcommand{\Mod}{\operatorname{Mod}}

\author[O. Hirviniemi]{Olli Hirviniemi}
\address{University of Helsinki, Department of Mathematics and Statistics, P.O. Box 68, FIN-00014 University of Helsinki, Finland}
\email{olli.hirviniemi@helsinki.fi}

\author[L. Hitruhin]{Lauri Hitruhin}
\address{Department of Mathematics and Systems Analysis Aalto University,
P.O. Box 11100 FI-00076 Aalto
Helsinki, Finland}
\email{lauri.hitruhin@aalto.fi}

\author[I. Prause]{Istv\'an Prause}
\address{Department of Physics and Mathematics, University of Eastern Finland, P.O. Box 111, 80101 Joensuu, Finland}
\email{istvan.prause@uef.fi}

\author[E. Saksman]{Eero Saksman}
\address{University of Helsinki, Department of Mathematics and Statistics, P.O. Box 68, FIN-00014 University of Helsinki, Finland}
\email{eero.saksman@helsinki.fi}

\title{On mappings of finite distortion that are quasiconformal in the unit disk}
\begin{document}

\begin{abstract}
We study quasiconformal mappings of the unit disk that have homeomorphic planar extension with controlled distortion. For these mappings we prove a  bound for the modulus of continuity of the inverse map, which somewhat surprisingly is almost as good as for global quasiconformal maps. Furthermore, we give examples which improve the known bounds for the  three point property of generalized quasidisks. 
  Finally, we establish optimal regularity of such maps when the image of the unit disk has cusp type singularities.  
\end{abstract}

\maketitle
\section{Introduction}
In recent years there has been a considerable interest in studying mappings that are conformal or quasiconformal on the unit disk, but admit a controlled extension to a self-map of the plane. This includes regularity properties of conformal maps which admit quasiconformal extension, see e.g. \cite{AN,PS,Pr19}, as well as study of modulus of continuity of such maps when the extension is a homeomorphic map of finite distortion \cite{Zap}, and for sufficient or necessary three point conditions on the boundary which guarantee the existence of a controlled extension \cite{GUO1, GKT, GUOXU, KT}. One may also note that the kind of maps in question appear naturally when one uses quasiconformal approach to rough conformal welding  problems. In the present paper we continue this line of research improving and complementing some previous results  of the references above.

\subsection{Three point condition.} 
From the work of Gehring in \cite{Geh} it is known that so-called Ahlfors three point condition completely describes when a conformal mapping from the unit disk can be extended to planar quasiconformal mapping. It is clear that if we relax the condition on the distortion of the extension, then also the three point condition can be relaxed, which leads to a three point condition with a \emph{non-linear} control function, see Section \ref{se:prereq}. This property has been recently studied by, among others, Guo and Xu \cite{GUO1, GUOXU}.  In cases when the homeomorphic extension is assumed to have exponentially integrable or  integrable distortion,  they find a sufficient three point condition at the boundary to ensure extendability.  On the other hand, the previous best counterexamples for failure of extendability in these cases are attained by classical inner cusp constructions, see \cite{GKT, KT}, and leave a gap between the positive and negative results.   We improve these examples by using different target domains than inward cusps, and hence make the gaps substantially smaller. First we consider the case with exponentially integrable distortion.
\begin{theorem} \label{3point}
	Fix any $K \geq 1$ and $\gamma >0$. There exists a Jordan domain $\Omega$ whose boundary satisfies three point condition with a control function
	\begin{equation}\label{controllexpint}
	h(t)=Ct \log^{\frac{1}{2}+\gamma}\left(\frac{1}{t}\right),
	\end{equation}
	such that there does not exist a $K$-quasiconformal mapping  $f: \D \to \Omega$ possessing a planar homeomorphic extension with exponentially integrable distortion.
\end{theorem}

In \cite{GKT} the authors used inward cusps as target domains to prove Theorem \ref{3point} with the control function $h(t)=Ct \log^{1+\gamma}\left(\frac{1}{t}\right)$, and thus we improve the exponent in \eqref{controllexpint} by the factor of $\frac{1}{2}$. For the positive direction Guo and  Xu  in 
\cite[Theorem 3.3]{GUOXU}  show that if the three point condition holds for $\Omega$ with the gauge  $h(t)=Ct [\log \log \left( e + \frac{1}{t} \right)]^{1/2-\varepsilon} $, then any  conformal map $f:\D\to \Omega$ has  planar extension with exponentially integrable distortion.\footnote{We mention  in passing that if the proof of  \cite[Theorem 1.1]{GUO1} could be corrected (see \cite[Remark 3.4]{GUOXU}) this would show that the exponent in Theorem \ref{3point} is sharp.}
 %We are not certain if his original proof can be corrected, but if so we would have complete understanding of three point property for extensions with exponentially integrable distortion.   

We prove a similar result when we only assume that the distortion of extension is locally $p$-integrable.
\begin{theorem}\label{Teoreema2}
	There exists a Jordan domain $\Omega$ whose boundary satisfies three point condition with the control function $h(t)=Ct^s$, where $s\in (0,1)$, and for which there does not exist a conformal mapping $f$ with $f(\D)= \Omega$ that has a planar homeomorphic extension with $p$-integrable  distortion if $p \geq 1$ and
	\begin{equation}\label{controllpint}
	p >\frac{s}{2(1-s)}.
	\end{equation} 
\end{theorem}
Here the positive result  obtained by Guo in \cite{GUO1} shows that if a domain $\Omega$ satisfies three point condition with the control function $h(t)=Ct^s$ we can always find a conformal mapping $f:\D \to \Omega$ with planar extension that satisfies $K_{f}^{p} \in L^{1}_{\text{loc}}$ when $$p < \frac{s^2}{2(1-s^2)}.$$ In the other direction it was shown in \cite{GKT}, using inward cusps, that Theorem \ref{Teoreema2} holds for the choice $$p> \frac{s}{1-s}.$$ Thus we again improve the gap in the exponent by the factor $\frac{1}{2}$. 

\subsection{Modulus of continuity for the inverse and bounds for the rotation}
We investigate boundary behaviour from the point of view of modulus of continuity. When the extension is assumed to have $p$-exponentially integrable distortion (see Section \ref{se:prereq} for the definition),  Zapadinskaya in \cite{Zap} proved the optimal upper bound 
\begin{equation}\label{UppermodulusZapadinskay}
|f(z)-f(w)| \leq \frac{C}{\log^{p}\left( \frac{1}{|z-w|} \right)}
\end{equation}
for the modulus of continuity. We prove a counterpart to this result by establishing the lower bound.
\begin{theorem}\label{LowerModCont}
	Let $f$ be a homeomorphism with $p$-exponentially integrable distortion which is $K$-qua\-si\-con\-for\-mal in the unit disk. Then for all $\varepsilon >0$ there exists a constant $c>0$ such that
	\begin{equation}\label{LowerModContyht}
	|f(z)-f(w)| \geq c|z-w|^{2K(1+\varepsilon)} \qquad \text{for any } z,w\in \overline{ \D}.
	\end{equation} 
	
\end{theorem}

We note that Zapadinskaya's upper bound \eqref{UppermodulusZapadinskay} is very close to the classical upper bound $|f(z)-f(w)| \leq C\log^{-p/2}( \frac{1}{|z-w|})$ of the planar mappings with exponentially integrable distortion, see \cite{OnniXiao}, while the lower bound \eqref{LowerModContyht} is surprisingly close to the bound $$|f(z)-f(w)| \geq C|z-w|^{K}$$ for planar $K$-quasiconformal maps. This highlights the different behaviour at outward cusps, which attain the extremality of the upper bound, and inward cusps giving rise to the lower bound. 

%\subsection{Rotation estimates.}
As explained in, for example, \cite{AIPS, H2} stretching and rotation are closely intertwined with each other. So it is not surprising that we can use Theorem \ref{LowerModCont} to establish control for pointwise rotation of these maps. Here we will use the upper half-plane instead of the unit disk to make notation simpler.
\begin{corollary}\label{co:rotation}
	Let $f$ be a planar homeomorphism with $p$-exponentially integrable distortion which is $K$-quasiconformal in the upper half-plane and normalized by $f(0)=0$ and $f(1)=1$. Then we have
	\begin{equation}\label{Rotationbound}
	|\arg (f(z))| \leq c  K  \log\left( \frac{1}{|z|} \right) \qquad \text{when $z \in \mathbb{R}$ is small} ,
	\end{equation} 
	where $c$ is a constant independent of any other parameters. 
\end{corollary}
For planar $K$-quasiconformal mappings the optimal pointwise bound for rotation was shown in \cite{AIPS} to be 
\begin{equation}\label{rotationkvasi}
|\arg (f(z))| \leq \left( K - \frac{1}{K} \right)  \log\left( \frac{1}{|z|} \right)+o \left(\log \left( \frac{1}{|z|} \right) \right),
\end{equation}
so for any fixed $K>1$ we get an analogous result to stretching and give a bound that is surprisingly close to that of planar quasiconformal mappings.

In the first part of  Section \ref{se:modulus} we  provide rather general modulus estimates for annuli in relation to quasiconformal maps with $\psi (K_f)$ integrable, where $\psi$ is a given convex gauge function, see Theorem \ref{le:3} and Corollary \ref{cor:2} below. These estimates are used in the proof of Theorem \ref{LowerModCont}. The more general modulus estimates and the proof of  Theorem \ref{LowerModCont} can be  combined to treat  mappings with sub-exponentially integrable distortion. Corollary \ref{cor:subexp} writes down explicitly the obtained bound for the class of maps with $\exp\Big(\frac{pK_f}{1+\log K_f}\Big)\in L^1_{loc} $. Again the result improves in a drastic manner the known continuity bound for the inverses when no quasiconformality in $\D$ is assumed.

\subsection{Regularity at cusps.} Finally we also study the integrability  of the derivative of
conformal mappings of the unit disk that have a planar extension with  exponentially integrable distortion, and such that the domain $f(\D)$ is $C^{1+\varepsilon}$-regular outside outward cusps. We refer to Section \ref{se:regularity} for the precise definition of  $C^{1+\varepsilon}$-cusps.
\begin{theorem}\label{cor:exp_p_cusps}
Let $f:\D\to\C$ be a conformal map that extends to a homeomorphism of $p$-exponentially integrable distortion to a neighbourhood of $\D$. Assume that  $f(\D)$ has $C^{1+\varepsilon}$-regular boundary apart from finitely many points that are outward $C^{1+\varepsilon}$-cusps. Then
$$
\int_\D |f'(z)|^2 \log^{2p+ 1-\varepsilon}(e + |f'(z)|)dA(z)<\infty
$$
for any $\varepsilon >0$.
\end{theorem}

\subsection*{Acknowledgements.} We are grateful for the anonymous referee for careful reading of the paper and his or her thoughtful comments. The work was supported by the Finnish Academy Coe `Analysis and Dynamics',  ERC grant 834728 Quamap and the Finnish Academy projects {1309940 and 13316965}. 

\section{Prerequisites} \label{se:prereq}
Let  $f:\Omega_1 \to \Omega_2$ be a sense-preserving homeomorphism between planar domains $\Omega_1, \Omega_2 \subset \mathbb{C}$. We say that $f$ is a \emph{$K$-quasiconformal mapping} for some $K \geq 1$ if $f \in W^{1,2}_{\text{loc}}$ and  the distortion inequality
\begin{equation*}
|Df(z)|^2 \leq KJ_f (z)
\end{equation*}
is satisfied almost everywhere. Here 
\begin{equation*}
|Df(z)|=\max \{|Df(z) e|: e\in \mathbb{C}, |e|=1\},
\end{equation*}
whereas $J_f(z)$ is the Jacobian of the mapping $f$ at the point $z$. 

Furthermore, we  say that $f$ has \emph{finite distortion} if the following conditions hold:
\begin{itemize}
	\item $f\in W_{\text{loc}}^{1,1}(\Omega_1)$
	\item $J_f(z)\in L^{1}_{\text{loc}}(\Omega_1)$
	\item $|Df(z)|^2\leq K(z)J_f(z) \qquad \text{almost everywhere in $\Omega_1$},$
\end{itemize}
for a measurable function $K(z)\geq 1$, which is finite almost everywhere. The smallest such function is denoted by $K_f(z)$ and called the distortion of $f$. Generally speaking mappings of finite distortion in their full generality  have too wild behaviour and thus additional restrictions are usually needed.These restrictions are usually additional assumptions on the distortion function $K_f$. We say that a mapping of finite distortion has \emph{$p_1$-exponentially integrable distortion} with a parameter $p_1>0$ if 
\begin{equation*}
e^{p_1 K_{f}(z) } \in L^{1}_{\text{loc}} 
\end{equation*}
and \emph{$p_2$-integrable distortion} with parameter $p_2 \geq 1$ if 
\begin{equation*}
K_{f}(z)\in L_{\text{loc}}^{p_2}(\Omega).
\end{equation*} 
These two classes of mappings are most commonly studied, but one can, of course, use also other integrability conditions for the distortion. For a closer look on planar mappings of finite distortion see \cite{AIM}. In the present paper we only consider \textit{homeomorphic} mappings of finite distortion.
\medskip

We say that a  bounded Jordan domain $\Omega$ satisfies the \emph{three point condition} with an increasing control function $h: [0,\infty) \to [0, \infty)$ if there exists a constant $C\geq 1$ such that for each pair of points $x,y \in \partial \Omega$ we have
\begin{equation*}
\min_{i \in \{1,2\}} \text{diam}(\gamma_i) \leq h\left(C|x-y|\right),
\end{equation*} 
where $\gamma_1$ and $\gamma_2$ are the components of $\partial \Omega \setminus \{x,y\}$. 
\medskip

Let $f: \mathbb{C} \to \mathbb{C}$ be a mapping of finite distortion and fix a point $z_0 \in \mathbb{C}$. In order to study the pointwise rotation of $f$ at the point $z_0$, we fix an argument $\theta\in[0,2\pi)$, and then look at how the quantity
$$\arg (f(z_0+te^{i\theta})-f(z_0))$$ 
changes as the parameter $t$ goes from 1 to a small $r$. This can also be understood as the winding of the path $f\left( [z_0+re^{i\theta}, z_0+e^{i\theta}] \right)$ around the point $f(z_0)$. One can then either study the maximal pointwise spiraling, in which case we need to consider all directions $\theta$,
\begin{equation}\label{spiraling}
\sup_{\theta \in [0,2\pi)} |\arg (f(z_0+re^{i\theta})-f(z_0)) - \arg (f(z_0+e^{i\theta})-f(z_0))  |,
\end{equation}
or, as in Corollary \ref{co:rotation}, we can restrict ourselves to some fixed directions $\theta$ in the above limit. 

Whichever we choose, the maximal pointwise rotation is precisely the behavior of the above quantity \eqref{spiraling} when $r\to 0$. In this way, we say that the map $f$ \emph{spirals at the point $z_0$ with a rate function $g$}, where $g: [0, \infty) \to [0, \infty)$ is a decreasing continuous function, if 
\begin{equation}\label{Spiral rate}
\limsup_{r\to 0} \frac{\sup_{\theta \in [0,2\pi)} |\arg (f(z_0+re^{i\theta})-f(z_0)) - \arg (f(z_0+e^{i\theta})-f(z_0))  |}{g(r)} = C
\end{equation}
for some constant $C>0$. Finding maximal pointwise rotation for a given class of mappings equals finding the maximal spiraling rate for this class. Note that in \eqref{Spiral rate} we must use limit superior as the limit itself might not exist. Furthermore, for a given mapping $f$ there might be many different sequences $r_n \to 0$ along which it has profoundly different rotational behaviour. 
For more on these definitions we refer to  \cite{AIPS, H4}. 
\medskip 

In proofs of our theorems the \emph{modulus of path families} will play an important role. We provide here the main definitions, but interested reader can find a closer look at the topic in \cite{V}. The image of a line segment $I$ under a continuous mapping is called a path and we denote by $\Gamma$  a family of paths. Given a path family $\Gamma$ we say that a Borel measurable function $\rho:\mathbb{C} \to [0,\infty]$ is admissible with respect to  it if any locally rectifiable path $\gamma \in \Gamma$ satisfies
\begin{equation*}
\int_{\gamma} \rho(z) |dz| \geq 1.
\end{equation*}
Denote the modulus of the path family $\Gamma$ by $M(\Gamma)$ and define it by 
\begin{equation*}
M(\Gamma)= \inf_{\rho \text{ admissible}} \int_{\mathbb{C}} \rho^{2}(z) dA(z).
\end{equation*}
As an intuitive rule the modulus is large if the family $\Gamma$ has ``lots" of short paths, and is small if the paths are long and there is not ``many" of them. 

We will also need a \emph{weighted} version of the modulus. Let a weight function $\omega: \mathbb{C} \to [0, \infty)$, which in our case will always be the distortion function $K_f$, be  measurable and locally integrable.  Define the weighted modulus $M_{\omega}(\Gamma)$ by
\begin{equation*}
M_{\omega}(\Gamma)= \inf_{\rho \text{ admissible}} \int_{\mathbb{C}} \rho^{2}(z) \omega(z) dA(z).
\end{equation*}
Finally, we will need the modulus inequality
\begin{equation}\label{Modeq}
M(f(\Gamma)) \leq M_{K_f}(\Gamma)
\end{equation}
which holds for any mapping of finite distortion $f$ for which the distortion $K_f$ is locally integrable, see \cite{H3}. 
In Section \ref{se:modulus} we also introduce the notion of \emph{conditional modulus}.

\section{Three point property}
In this section we study the three point property for the image of the unit disk under  quasiconformal mappings that admit extension to a global mapping with controlled distortion. In particular, we prove Theorems \ref{3point} and \ref{Teoreema2}.

The example for both theorems is constructed by a similar principle: our domain $\Omega$ will be the unit disk with a `snake-like tunnel'  removed inside a sector with angle $\alpha$. The tunnel folds itself in a self-similar manner so that it essentially covers the whole sector. We will first concentrate on the proof of Theorem \ref{3point}, and then  the widths of both the tunnel and its complement at the distance $r$ from the origin are chosen to be comparable to
$$
 \frac{r}{\log^{\frac{1}{2}+\varepsilon}(1+\frac{1}{r})}.
$$ 
%where $C_1<\tilde{C}<C_2$ for some fixed $C_1,C_2$. %The actual value of the $\tilde{C}$ is allowed to vary slightly to make the construction of the slice simpler. 
Furthermore, apart from the small ``turning areas'', that have negligible area when $r$ is small, the snake consists of vertical tubes. See Figure \ref{Snake} for an illustration of $\Omega$.

	\begin{figure}[ht]	
		\begin{center}
			\includegraphics[scale=0.6]{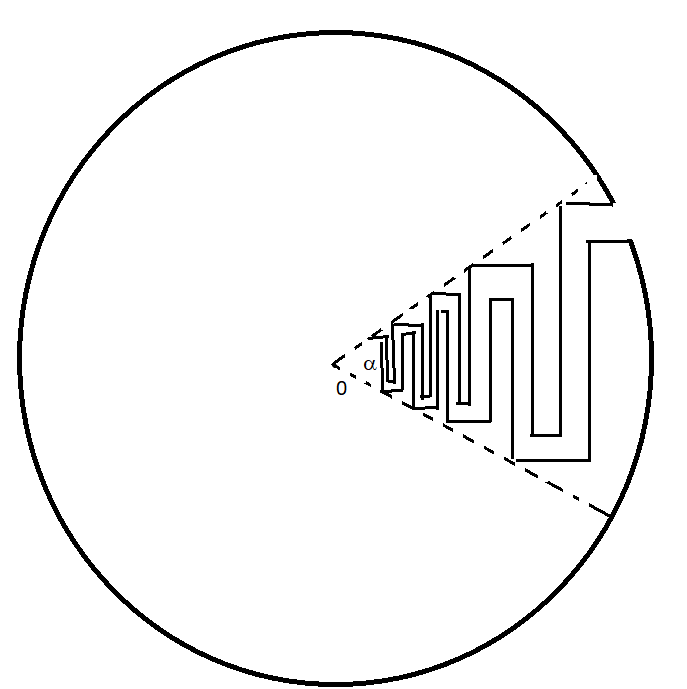}
			\label{Kvasikuva}
		\end{center}
		\caption{Domain $\Omega$ with snake-like tunnel}
		\label{Snake}
	\end{figure}
	
	We will first  show that there does not exist a $K$-quasiconformal mapping $f$ such that $f(\D)=\Omega$, for any choice of $\varepsilon >0$, which has extension with exponentially integrable distortion. As a second step we will then verify that these domains $\Omega$ satisfy the required 3-point property. 
	
	Towards our first goal, we next   bound the diameter of the pre-image of a small piece of the snake like tunnel containing  the tip (i.e. the origin). For a given $\overline{r}>0$, let $A_{\overline{r}}$ be the maximal initial part of the tunnel starting from the tip (i.e. maximal subdomain with respect to inclusion) that is contained in $B(0,\overline{r})$ and is obtained by a single crosscut along the real axis.
	
	\begin{lemma}\label{apulemma}
	Let $A_{\overline{r}}$ and $\Omega$ be as above.	Assume that $f: \D \to \Omega$ is $K$-quasiconformal and denote also by $f$ the continuous extension to the boundary which is guaranteed to exist by Carath\'eodory's theorem.  Denote $E_{\overline{r}} =  \partial A_{\overline{r}} \cap \partial \Omega$. Then  
		\begin{equation}
		\mathrm{diam} \, f^{-1}(E_{\overline{r}})  \geq C_f \left( \mathrm{diam} \, E_{\overline{r}}  \right)^{\frac{K}{2-\varepsilon_\alpha}} \geq C \overline{r}^{\frac{K}{2-\varepsilon_\alpha}},
		\end{equation}
		for all $\overline{r}$ small enough. Furthermore, $\varepsilon_\alpha \to 0$ when $\alpha \to 0$. 
	\end{lemma}
	
	\begin{proof}
	We will apply modulus of path families and  follow rather closely the proof of \cite[Lemma 4.2]{GKT}. So, let $U$ and $F$ be the two components of the boundary of the tunnel that remain  after removing the part $E_{\overline{r}}$ from the tunnel. Let $\Gamma$ be the family of paths $\gamma \subset \Omega$ connecting $U$ and $F$. The estimate for the modulus of $\Gamma^{-1}:=f^{-1} \left( \Gamma \right)$ is rather standard, but we provide the details for the reader's convenience. 
	
	That is, we start by choosing point $\overline{z} \in \partial \D$ such that $\dist(\overline{z},f^{-1}(U))=\dist(\overline{z},f^{-1}(F))=a$. Note that this distance is essentially $\frac{\text{diam}  f^{-1}(E_{\overline{r}})}{2}$ and denote by $A$ the minimum of the distances between $\overline{z}$ and the other end point of $f^{-1}(U)$ and $f^{-1}(F)$. Clearly $A$ does not depend on $\overline{r}$ and is just a constant bounded from above by 2.   
	
	For each $t$ with $a<t<A$ we see that the circle $S(\overline{z},t)$ intersects both $f^{-1}(U)$ and $f^{-1}(F)$. Thus for any admissible $\rho$ inside the unit disk we have
	\begin{equation*}
	1 \leq \left( \int_{S(\overline{z},t)} \rho \, |dz| \right)^2 \leq r \pi \int_{S(\overline{z},t)} \rho^2\,  |dz|
	\end{equation*}
	Thus we can use Fubini's theorem to obtain  
	\begin{equation}\label{Alaraja}
	M(\Gamma^{-1}) \geq  \int_{a}^{A} \int_{S(\overline{z},t)} \rho^2 \, [dz|dt = \int_{a}^{A} \frac{1}{\pi t} dt \geq  \frac{\log\left( \frac{C}{\text{diam}  f^{-1}(E_{\overline{r}})  } \right)}{\pi},
	\end{equation}
 To estimate the modulus $M(\Gamma)$ from above define
	\begin{equation*}
	\rho_{0}(z) = 
	\begin{cases}
	\frac{1}{(2\pi - \alpha)|z|} &\quad\text{if $ \overline{r}/2e^{2\pi} \leq |z|\leq 1$ and $z\in\Omega$,}\\
	0 &\quad\text{otherwise}. \\
	\end{cases}
	\end{equation*}
	We claim that $\rho_0$ is admissible for the path family $\Gamma$. To see this we look at two possibilities. In the first alternative the curve $\gamma\in\Gamma$  intersects the circle  $S(0,\overline{r}/2e^{2\pi})$. Then, just observing the radial variation and noting that $\int_{\overline{r}/2e^{2\pi}}^{\overline{r}/2}\frac{dr}{r}=2\pi>2\pi - \alpha$ we obtain that $\rho$ is admissible for $\gamma.$ In the second case $\gamma$ stays in the simply connected domain 
	$\Omega\cap\{\overline{r}/ 2e^{2\pi}<|z|<1\}$, where we may define a branch of $\arg(z)$. Thus, letting  $\gamma$ be parametrized on $[0,1]$  we obtain 
	$$
	\int_\gamma \rho|dz|\geq (2\pi-\alpha)^{-1}\int_\gamma |d(arg(z))|\geq (2\pi-\alpha)^{-1}\big|\arg(\gamma(1))-\arg(\gamma(0))|\geq 1,
	$$ and again $\rho$ is admissible for $\gamma.$ 

	Thus we can use $\rho_0$ to estimate  
	\begin{equation}\label{ylaraja}
	M(\Gamma)  \leq \int \rho_{0}^{2} \, dA(z)\leq \frac{1}{2\pi-2\alpha}\log\left( \frac{C}{\overline{r}} \right).
	\end{equation}
	Since $f$ is $K$-quasiconformal in the unit disk we have
	\begin{equation*}
	M(\Gamma^{-1})\leq K M(\Gamma),
	\end{equation*}
	and the claim follows by combining  \eqref{Alaraja} with \eqref{ylaraja}.\end{proof}
	
\subsection{Proof of Theorem \ref{3point}}	
	Assume we have a planar mapping $f$ that is $K$-quasiconformal inside $\D$, has $p$-exponentially integrable distortion and $f(\D)=\Omega$. Fix $\alpha$ which gives the width of the snake (this angle actually plays no role here but will be important for $p$-integrable extension case). Then by Lemma \ref{apulemma} we have that 
	\begin{equation}\label{aputulos}
	\text{diam}  f^{-1}(E_{\overline{r}})  \geq C \overline{r}^{\frac{K}{2-\varepsilon_\alpha}}
	\end{equation}
	whenever $\overline{r}$ is small and $E_{\overline{r}}$ is the $\overline{r}$-neighbourhood of the tip of the tunnel as before. Let also $U$ and $F$ be as before the remaining  components of the boundary of the tunnel. Let $\Gamma$ be the family of paths connecting the preimages of these components $U_0 :=f^{-1}\left(U \right)$ and $F_0 :=f^{-1}\left( F \right)$ in the exterior $\C\setminus\overline{\D}$ of the unit disk. We aim to use the general modulus inequality \eqref{Modeq}
	\begin{equation}\label{modey}
	M(f(\Gamma)) \leq M_{K_f}(\Gamma).
	\end{equation}
	with good estimates for both of these moduli. 
	
	\medskip
		We start with estimating $M_{K_f}(\Gamma)$ from above when $\overline{r}$ is small.  
	As we estimate modulus from above it is enough to find one suitable admissible function. To this end let 
	\begin{equation*}
	d= \dist (U_0 , F_0)= \diam f^{-1}(E_{\overline{r}})
	\end{equation*}
	and let $U_{0,i} \subset U_0$ be the set of points $z\in U_0$ satisfying $ 2^{i}d < \dist(z, F_0) \leq 2^{i-1}d$. Here $i$ goes from 1 to $n$, which is defined as the smallest integer such that $2^nd > \diam U_0$. As we can assume that sets $U_0$ and $F_0$ are small compared to the unit disk, by taking subsets of $U_0$ and $F_0$ if needed, sets $U_{0,i}$ are segments of $\partial\D$ with one end point having distance $2^{i-1}d$ to the set $F_0$ while the other end point has distance $2^{i}d$. Thus the sets $U_{0,i}$ are disjoint and cover the set $U_0$.
	
	Denote $S_i=\{z\in \C\setminus \overline{\D}: \dist(z,U_{0,i}) < 2^{i-1}d\}$ and define (with the notation $S_0 = \emptyset$)
	\begin{equation}\label{rho_0}
	\rho_{0}(z) = 
	\begin{cases}
%	\frac{1}{d} &\quad\text{if $ z \in S_1,$}\\
%	\frac{1}{2d}  &\quad\text{if $z \in S_2 \setminus S_1,$} \\
	\frac{1}{2^{i-1}d} & \quad \text{if $z \in S_i \setminus S_{i-1},$} \\
%	\frac{1}{2^{n-1}d} & \quad \text{if $z \in S_n \setminus S_{n-1},$} \\
	0 & \quad \text{otherwise.}
	\end{cases}
	\end{equation}
	The function $\rho_0$ is admissible for $\Gamma$ as $S_i \cap F_0 = \emptyset $ for every $i$ and $\rho_0(z) \geq \frac{1}{2^{i-1}d}$ when $z\in S_i$.

	With an admissible function at hand we use the pointwise estimate 
	\begin{equation*}
	\frac{1}{p}(K_f p)\rho_{0}^{2} \leq \frac{1}{p} \left( e^{p K_f} -1\right)+\frac{1}{p} \rho_{0}^{2} \log(1+\rho_{0}^{2})
	\end{equation*}
	to get rid of the weight. This leads to an estimate
	\begin{equation*}
	M_{K_f}(\Gamma) \leq \int K_f(z) \rho_{0}^{2}(z)  dA(z) \leq \frac{1}{p}\int_{B(0,5)} e^{pK_f(z)}  dA(z) \quad + \frac{1}{p} \int \rho_{0}^{2} \log(1+\rho_{0}^{2})  dA(z). 
	\end{equation*}
	The first integral is bounded above by a constant $\alpha_f$ due to the assumption on the distortion, so we can concentrate on the second integral. First note  
	\begin{equation*}
	\log(1+\rho_{0}^{2}) \leq 3 \log\left( \frac{1}{d} \right) \leq C_K \log \left( \frac{1}{\overline{r}} \right)
	\end{equation*}
	since 
	\begin{equation*}
	\dist(U_0,F_0)=d \geq \overline{r}^{C_K},
	\end{equation*}
	by Lemma \ref{apulemma}, where $K$ is the quasiconformality constant inside $\D$. 
	
		To estimate the rest of the integral we note that the area of $S_i$ is bounded from above by $C \left(2^{i-1}d\right)^2$ for each $i$ with fixed $C$. This allows us to estimate
	\begin{equation*}
	\int \rho_{0}^{2}  dA(z) \leq \sum_{i=1}^{n} \int_{S_i} \left(\frac{1}{2^{i-1}d}\right)^2  dA(z) \leq Cn,
	\end{equation*}
	where $n$ was chosen to be the smallest integer such that $2^nd > \diam U_0$. With the trivial estimate $\diam U_0 \leq 1$ this leads to (we denote by $C_K$ a constant that depends only on $K$, but whose value may otherwise change from line to line)
	\begin{equation*}
	n \leq c \log\left( \frac{1}{d} \right) \leq C_K \log\left( \frac{1}{\overline{r}} \right) 
	\end{equation*}
	giving  
	\begin{equation*}
	\int \rho_{0}^{2}  dA(z) \leq C_K \log\left( \frac{1}{\overline{r}} \right) 
	\end{equation*}
	where $K$ is, as before, the quasiconformality constant inside $\D$. Thus 
	\begin{equation}\label{modulusylhaalta}
	M_{K_f}(\Gamma) \leq \alpha_f + C_K \log^2 \left( \frac{1}{\overline{r}} \right).
	\end{equation}
	This finishes the estimate from above for the weighted moduli. 
	
	\medskip
	In order to use the modulus inequality \eqref{modey} we must next estimate $M(f(\Gamma))$ from below. To this end we note that the $n$:th vertical part of the tunnel (counting from right), with left boundary having value $r_n$ on $x$-axis, has width $\tilde{C}r_n/\log^{\frac{1}{2}+\varepsilon}(1+\frac{1}{r_n})$, where $C_1<\tilde{C}<C_2$ for some fixed $C_1,C_2$. Let us call $R_n$ the  $n$:th vertical part the `turning area' excluded.  A standard application of Cauchy-Schwarz inequality (exactly the same reasoning as when computing the modulus of a rectangle) shows that for any admissible $\rho$  for the family $f(\Gamma)$  it holds that % the integral $\int_{R_n}\rho^2$ is lower bounded by
        \begin{equation}\label{admissible_r}
	\int_{R_n}\rho^2\gtrsim |R_n| r_n^{-2}  \log^{1+ 2\varepsilon}\left( 1+ 1/r_n \right)
	\end{equation}
We clearly have 	$ r_n^{-2}\log^{1+ 2\varepsilon}\left( 1+ 1/r_n \right)\approx |z|^{-2} \log^{1+ 2\varepsilon}\left( 1+ 1/|z| \right)$ for $z\in R_n$. We denote by $A_\alpha:=\{z\in\D\;:\; |\arg(z)|\leq\alpha/2\}$ the angular slice of $\D$ enveloping the tunnel, and get directly by construction of the tunnel
	\begin{eqnarray}\label{alamod}
	M(f(\Gamma))& \gtrsim &\int_{\cup_n R_n} |z|^{-2}\log^{1+ 2\varepsilon}\left(1+ 1/|z|\right) dA(z) \\
	& \geq & \int_{\{|z|\geq\overline{r}\}\cap A_\alpha}|z|^{-2}\log^{1+ 2\varepsilon}\left(1+ 1/|z|\right) dA(z) \nonumber\\
	& \gtrsim & C_\alpha\int_{\overline{r}< |z|<1}|z|^{-2}\log^{1+ 2\varepsilon}\left(1+ 1/|z|\right) dA(z)\nonumber \\
	& \gtrsim & C_{\alpha} \log^{2+2\varepsilon}\left( \frac{1}{\overline{r}} \right).\nonumber
	\end{eqnarray}

	Now combining the estimates for moduli with the modulus inequality \eqref{modey} we get a contradiction when $\overline{r} \to 0$, and hence there cannot be a mapping $f$ with exponentially integrable distortion for any $\varepsilon>0$. 
	
	\medskip
	As a final step we show that the domain $\Omega$, with any choice for $\varepsilon>0$, satisfies the 3-point condition with gauge function
	\begin{equation*}
	h(t)=\overline{C}t \log^{\frac{1}{2}+\varepsilon+\gamma}(\frac{1}{t})
	\end{equation*}
	for any $\gamma>0$. 
	
	This is trivial outside of the snake like tunnel. So, let us choose $z_0$ in snake with distance $r$ from the origin. From the construction of $\Omega$ we see that there exists fixed $C>0$ so that the intersection $$B\left(z_0,Cr/\log^{\frac{1}{2}+\varepsilon}\left(1+\frac{1}{r}\right)\right) \cap \partial \Omega$$ consists of exactly one continuum, for which 3-point condition holds trivially. Thus we need to check the condition only for points $w$ not in this ball. In this case we can clearly by the construction of the tunnel  assume that
	\begin{equation*}
	|z-w|=C_1r/\log^{\frac{1}{2}+\varepsilon}\left(1+\frac{1}{r}\right),
	\end{equation*}
	where $C_1>0$, and we need to show that 
	\begin{equation*}
	h(|z-w|)\geq C_2r,
	\end{equation*}
	since the diameter of the part of the boundary of the  tunnel connecting $z$ and $w$ is $C_2r$, where $C_2$ is controlled from above and below. But a direct computation gives 
	\begin{equation*}
	h(|z-w|) \geq \overline{C} C_1 r \log^{\gamma} \left( 1+\frac{1}{r} \right),
	\end{equation*}
	which is clearly larger than $C_2 r$,
	 when $\overline{C}$ is large enough. This finishes the proof of Theorem \ref{3point}. 

\smallskip	 
	 
	\noindent\textbf{Remark.}  It would be interesting to consider examples with spiralling tunnels; however our modulus estimates do not directly apply in such a situation.
	
\subsection{Proof of Theorem \ref{Teoreema2}}	
In this case we assume only $p$-integrable distortion from the extension. The proof follows closely the proof of Theorem \ref{3point} since we are using the same basic construction and we just concentrate on the modifications  needed. We use a similar domain that has a snake like tunnel as in the proof of Theorem \ref{3point}, but this time we fix $s\in (0,1)$ and choose the width of  the tunnel at distance $r$ to the origin to be $\approx r^{\frac{1}{s}}$. We will again use the modulus inequality  \eqref{modey}
 using  the same choice for $\Gamma$. That is, as before let $U$ and $F$ be the sides of the snake like slice ending at distance $\overline{r}$. Then $\Gamma$ is the family of paths connecting the sets $U_0: = f^{-1}\left( U\right)$ and $F_0: = f^{-1}\left( F \right)$. Dependency on the modulus inequality is the reason for the assumption that $p\geq 1$ in Theorem \ref{Teoreema2}. 
 
Again we start with estimating $M_{K_f}(\Gamma)$ from above but this time we must consider cases $p>1$ and $p=1$ separately. Let us first assume $p>1$ and  denote 
\begin{equation*}
d:= \dist (U_0,F_0)= \diam f^{-1}(E_{\overline{r}}),
\end{equation*}
define the sets $U_{0,i}$, $S_i$ as before and use the same admissible function $\rho_0$ as before by the formula
\eqref{rho_0}.
Recall also that $n$ is chosen as the smallest integer such that $2^n d \geq 1$. 

We estimate using the H\"older inequality 
\begin{equation*}
M_{K_f}(\Gamma)\leq \int K_f(z) \rho_{0}^{2}  dA(z) \leq \left( \int_{B(0,5)} K_{f}^{p}  dA(z) \right)^{\frac{1}{p}} \left( \int \rho_{0}^{\frac{2p}{p-1}}  dA(z)\right)^{\frac{p-1}{p}}.
\end{equation*}
Here the first integral is bounded from above by a constant $\alpha_{f,p}$ so we can concentrate on the second integral.
To this end note that
\begin{equation*}
\left( \int \rho_{0}^{\frac{2p}{p-1}}  dA(z)\right)^{\frac{p-1}{p}} \leq \left( \sum_{i=1}^{n} \int_{S_i} \left( \frac{1}{2^{i-1}d} \right)^{\frac{2p}{p-1}}  dA(z)\right)^{\frac{p-1}{p}}.
\end{equation*}
And as we have an upper bound $C \left(2^{i-1}d\right)^2$ for the area of $S_i$ we get an upper bound
\begin{equation*}
\left( C_p d^{\frac{-2}{p-1}} \sum_{i=1}^{n} \frac{1}{2^{\frac{2(i-1)}{p-1}}} \right)^{\frac{p-1}{p}}.
\end{equation*}
The sum 
\begin{equation*}
\sum_{i=1}^{\infty} \left( \frac{1}{2^{\frac{2}{p-1}}} \right)^{i-1}
\end{equation*}
is bounded for any $p>1$ with a bound that depends only on $p$, and thus we can continue the estimate by
\begin{equation*}\label{modylärajapint}
M_{K_f}(\Gamma) \leq C_{p,K_f} \left(\frac{1}{d} \right)^{\frac{2}{p}}.
\end{equation*}
From Lemma \ref{apulemma}, which trivially holds also in this case as nothing in the proof changes when we make tunnel thinner, we have 
\begin{equation*}
d \geq C_f \left( \mathrm{diam} \, E_{\overline{r}}  \right)^{\frac{1}{2-\varepsilon_\alpha}} \geq C \overline{r}^{\frac{1}{2-\varepsilon_\alpha}},
\end{equation*}
where we can choose $\varepsilon_{\alpha}$ as small as we wish by making the angle $\alpha$ small enough. And since width of the tunnel goes to zero in polynomial speed we have no difficulty in choosing $\alpha$ arbitrary small as we still have enough room for the snake to turn. By substituting this to the estimate we get
\begin{equation}\label{modylärajapint1}
M_{K_f}(\Gamma) \leq C_{p,K_f} \left( \frac{1}{\overline{r}} \right)^{\frac{1+ \varepsilon_{\alpha}}{p}}.
\end{equation}
In the case $p=1$ we use admissible function
\begin{equation*}
\rho_{0}(z) = 
\begin{cases}
\frac{1}{d} &\quad\text{if dist $( z,U_0 )<d,$}\\
0  &\quad\text{otherwise}.
\end{cases}
\end{equation*}
to get bound
\begin{equation}\label{modylärajapintKUNP=1}
M_{K_f}(\Gamma) \leq \int K_f(z) \rho_{0}^{2} (z)  dA(z) \leq \frac{1}{d^2} \int_{B(0,5)} K_f(z)  dA(z) \leq \frac{C_{K_f}}{d^2} \leq C_{K_f}\left(\frac{1}{\overline{r}} \right)^{1+ \varepsilon_{\alpha}}.
\end{equation}

As the next step we estimate the modulus $M(f(\Gamma))$ from below. Here the method is exactly same as in the proof of Theorem \ref{3point} but the width at the distance $r$ was chosen to be $\tilde{C}r^{\frac{1}{s}}$, which yields as before the  estimate of the modulus in vertical parts $R_n$ to be estimated  from below by the corresponding integral of the function $g(z)=\frac{C}{|z|^{\frac{2}{s}}}$. And as before the integral of $g$ over vertical parts is up to a constant same as the integral of $g$ over segment with angle $\alpha$, inner radius $\overline{r}$ and outer radius $1$. This time we obtain 
\begin{equation*}
M(f(\Gamma)) \geq C_{\alpha} \int_{B(0,1)\setminus B(0, \overline{r})} \frac{C}{|z|^{\frac{2}{s}}}  dA(z),
\end{equation*}
where the constant $C_{\alpha}$ can be arbitrary small for small $\alpha$, but does not depend on $\overline{r}$. Calculating the integral we get 
\begin{equation}\label{alamodkunpint}
M(f(\Gamma)) \geq C_{\alpha} \left( \frac{1}{\overline{r}} \right)^{\frac{2-2s}{s}}.
\end{equation}
Combining estimate \eqref{alamodkunpint} with \eqref{modylärajapint1} (or with \eqref{modylärajapintKUNP=1} if $p=1$) and the modulus inequality \eqref{modey} we get the desired bound 
\begin{equation*}
p \leq \frac{s}{2(1-s)}+ \varepsilon_{\alpha},
\end{equation*}
where $\varepsilon_\alpha$ is not strictly same as before but still $\varepsilon_\alpha \to 0$ when $\alpha \to 0$, when we let $\overline{r} \to 0$. 

The proof of Theorem \ref{Teoreema2} is finished once we note that for any fixed $s\in (0,1)$ the domain $\Omega$ satisfies 3-point condition with control function $h(t)=t^{s}$. This is easy to see with similar consideration as before since the width of the snake at the distance $r$ is comparable to $ r^{\frac{1}{s}}$.

\section{Modulus of continuity for the inverse}
\label{se:modulus}

In this section we continue studying the boundary behaviour in the form of lower bounds for the modulus of continuity proving Theorem \ref{LowerModCont} and its Corollary \ref{co:rotation}. Moreover, we will improve on Theorem \ref{LowerModCont} and describe a method that can establish lower bounds under more relaxed assumptions for the distortion. To this end, we start by building more general machinery for modulus estimates. 

\subsection{General estimates for modulus of rings}
First we couple the weighted modulus in an annulus $A$ with integrability of the distortion by defining the \emph{conditional modulus} $M_{\psi,I}(\Gamma)$ as follows:

\medskip
\framebox{$M_{\psi,I}(\Gamma)$ is the supremum of $M_K(\Gamma)$ under the condition that $\int_A\psi(K)\leq I$,}

\medskip
and write a  counterpart of the classical modulus of ring estimate using $M_{\psi,I}(\Gamma)$. 
\begin{lemma}\label{le:1}
	Let  $0\leq m<n$  be integers and denote by $\Gamma$  the family of curves in the annulus $A:=\{e^{-n}<|x|<e^{-m}\}$ joining the inner and outer boundary. Let $\psi : [1,\infty)\to [0,\infty)$ be a strictly increasing and convex homeomorphism. Then, writing $I_0=I/2\pi$,
	$$
	\inf_{\sum_{j=m+1}^{n}b_j=1} \sum_{j=m+1}^{n}\big(\psi^{-1}\big(I_0b_je^{2(j+1)}\big)\big)^{-1}\;\leq\;2\pi \big(M_{\psi,I}(\Gamma)\big)^{-1}\;\leq \;\inf_{\sum_{j=m+1}^{n}b_j=1} \sum_{j=m+1}^{n}\big(\psi^{-1}\big(I_0b_je^{2j}\big)\big)^{-1},
	$$
	where the numbers $b_j$ are assumed to be strictly positive. The difference between the left and right hand estimates is at most 1.
	
\end{lemma}
\begin{proof} We start by considering first a given radially symmetric distortion $K=K(|z|)\geq 1$ on the annulus $A$. All rotations of any admissible $\rho$ are then also admissible, and therefore so are their rotational averages. By convexity of the weighted $L^2$-norm, the rotational average has smaller weighted norm than the original admissible function, and hence the function $\rho$ giving minimum for the modulus is also radially symmetric, $\rho=\rho(|z|)$. Thus we look for
	\begin{equation}\label{eq:min}
	\min\;\; 2\pi\int_{e^{-n}}^{e^{-m}}\rho(u)^2K(u)udu \qquad\textrm{under the constraint}\qquad \int_{e^{-n}}^{e^{-m}}\rho(u)= 1
	\end{equation}
	(if a rotationally symmetric $\rho$ is admissible for the radial curves, it is admissible for all).  By Cauchy-Schwarz inequality 
	$$
	\Big(\int_{e^{-n}}^{e^{-m}}\rho(u)^2K(u)udu\Big)\Big( \int_{e^{-n}}^{e^{-m}}\frac{du}{K(u)u}\Big)\geq \Big(\int_{e^{-n}}^{e^{-m}}\rho(u)\Big)^2= 1,
	$$
	whence the minimum in \eqref{eq:min} equals
	%\begin{equation*}\label{eq:min2}
	$ =2\pi\Big( \int_{e^{-n}}^{e^{-m}}\frac{du}{K(u)u}\Big)^{-1}$.
	%\end{equation*}
	
	Recall that our  goal is to look for the largest modulus among distortions $K$ that satisfy the bound $\int_A\psi(K)\leq I$. We first treat the radially symmetric $K$, where 
	we thus need  to investigate the extremal problem
	\begin{equation}\label{eq:min3}
	\inf \int_{e^{-n}}^{e^{-m}}\frac{du}{K(u)u} \qquad\textrm{under the constraint}\qquad \int_{e^{-n}}^{e^{-m}}u\psi(K(u))du\leq I_0.
	\end{equation}
	Let us denote $a_j=\int_{e^{-j-1}}^{e^{-j}}\frac{du}{K(u)u}$, and note that \eqref{eq:min3} is essentially equivalent to
	\begin{equation}\label{eq:min4}
	\inf \sum_{j=m+1}^{n}a_j \qquad\textrm{under the constraint}\qquad \sum_{j=m+1}^{n}e^{-2j}\psi(1/a_j)\leq I_0.
	\end{equation}
	More precisely: the constraint in  \eqref{eq:min4} implies the original constraint, and is implied by the original if $I_0$ is replaced by $e^2I_0$.
	To see this, first of all by direct calculation we see that
	$$
	e^{-2}e^{-2j}\int_{e^{-j-1}}^{e^{-j}}\frac{\psi(K(u))du}{u}\leq\int_{e^{-j-1}}^{e^{-j}}u\psi(K(u))du \leq e^{-2j}\int_{e^{-j-1}}^{e^{-j}}\frac{\psi(K(u))du}{u}.
	$$
	Secondly, the  map $y\mapsto \psi(1/y)$ is also convex since $\psi$ is increasing and convex, and we may apply Jensen's inequality on the function $y \to \psi(1/y)$ and the probability measure $du/u$ on the interval $(e^{-j-1}, e^{-j})$   to deduce that
	$$
	\int_{e^{-j-1}}^{e^{-j}}\frac{\psi(K(u))du}{u}\geq \psi(1/a_j),
	$$
	where equality can be attained by  choosing $K$ constant on the interval $(e^{-j-1},e^{-j})$ while preserving the value of $a_j$.
	
	We denote $b_j:=  e^{-2j}\psi(1/a_j)/I_0 $ so that $b_j$'s sum to at most 1, and we note that $a_j= \big(\psi^{-1}(I_0e^{2j}b_j)\big)^{-1}$. The claim on the difference between the upper and lower estimates follows by noting that since in any case the $a_j$:s are bounded from above by 1, multiplying $I$ by the constant $e^2$ produces only a bounded error. 
	Namely, let $f:(0,\infty)\to(0,A]$ be monotonic. If $I_0$ is replaced by $I_0'$, the value of sum of the form 
	$\sum_{n=a}^b f(I_0e^{2n})$ can change at most by $A\lceil\frac{1}{2}|\log (I_0/I'_0)|\rceil$ in general. In our application we may take $A=1$,  which will then produce the universal error term 1.

	Up to now we have assumed radial symmetry on $K$. If $K$ is not radially symmetric, let $\widetilde K$ be the rotational average
	$ \widetilde K(z):=\frac{1}{2\pi}\int_0^{2\pi}K(|z|e^{i\theta})$ and note that $\int_A\psi(\widetilde K)\leq  \int_A \psi(K) \leq I$. Above we noted that the minimal $\rho$ giving the modulus $M_{ \widetilde K(z)}(\Gamma)$ is rotationally symmetric. Then note, that for any rotationally symmetric $\rho$ we have $\int_A\widetilde K\rho^2=\int_AK\rho^2$, which shows that the weighted moduli satisfy $M_{K}(\Gamma) \leq M_{\widetilde K}(\Gamma)$ finishing the proof.
	
\end{proof}
One of the most important properties of modulus of path families is that the modulus over annulus $A:=\{e^{-n}<|x|<e^{-m}\}$ converges to zero when $n \to \infty$ for fixed $m$. As a corollary to Lemma \ref{le:1} we get the exact range of control functions $\psi$ for which this holds. Similar results have been discovered earlier, see for example \cite{KKMOZ, KoOn}.

\begin{corollary}\label{cor:1} 
	In the situation of the previous lemma,  for fixed $m$ and $I$, we have $$\lim_{n\to\infty}M_{\psi,I}(\Gamma)=0$$ if and only if
	\begin{equation}\label{eq:ehto1}
	\sum_{j=1}^{\infty}\big(\psi^{-1}\big(e^{2j}\big)\big)^{-1}=\infty,
	\end{equation}
	or equivalently
	\begin{equation}\label{eq:ehto2}
	\int_2^\infty \frac{\log(\psi(x))}{x^2}dx=\infty.
	\end{equation}
\end{corollary}

\begin{proof}
	Let us first note that by choosing $b_j=e^{-j}$ on the right hand side of Lemma \ref{le:1}  and replacing each $b_j$ by 1 on the left hand side , we obtain the crude estimate
	$$
	\sum_{j=m}^{n}\big(\psi^{-1}\big(I_0e^{2(j+1)}\big)\big)^{-1}\;\leq \;2\pi \big(M_{\psi,I}(\Gamma)\big)^{-1 }\leq  \sum_{j=m}^{n}\big(\psi^{-1}\big(I_0e^{j}\big)\big)^{-1}.
	$$
	Clearly  monotonicity of $\psi^{-1}$  implies that here both upper and lower estimates tend to $\infty$  simultaneously when $n \to \infty$, and this happens exactly when \eqref{eq:ehto1} holds. 
	
	Next, in order to verify
	the equivalence of the conditions, note that by approximation we may assume  that $\psi$ is differentiable. By monotonicity \eqref{eq:ehto1} is equivalent to 
	$$
	\infty = \int_{\frac{1}{2}\log 2}^\infty (\psi^{-1}(e^{2t}))^{-1}dt=\frac{1}{2}\int_2^\infty \frac{dt}{t\psi^{-1}(t)}= c+\frac{1}{2}\int_2^\infty \frac{\psi'(x)}{x\psi(x)}.
	$$
	The last written integral  over the interval $[2,M]$ equals
	$$
	\Big/^M_2\frac{\log(\psi(x))}{x}+\int_2^M{\log(\psi(x))}\frac{dx}{x^2},
	$$
	which yields the desired conclusion since if  $\frac{\log(\psi(x))}{x}$ is unbounded as $x\to\infty$ we easily see that the last written integral  above diverges.
\end{proof}

The previous result shows that we can use the modulus to estimate e.g. continuity essentially only in the case where $\int_2^\infty \frac{\log(\psi(x))}{x^2}dx=\infty,$ which is naturally to be expected since this is well known be  the  borderline for cavitation in the theory of mapping of finite distortions, see \cite[Chapter 20.3]{AIM}.  For example, mappings of $L^p$-integrable distortion with $p<\infty$ do not work. Nontrivial estimate are still obtained e.g. in the cases 
$$
\psi(x)=\exp \Big(\frac{x}{\log(e+x)}\Big) -c_0  \qquad\textrm{or}\qquad  \psi(x)=\exp \Big(\frac{x}{\log(e+x)\log\log(e^2+x)}\big)-c_1. 
$$

We next state our main result concerning modulus estimates, which lets us estimate rather accurately the modulus of annuli for a large class of $\psi$'s which, roughly speaking, grow slower than double exponential: the condition of the following result is satisfied e.g. if $\psi(x)\leq \exp(\exp(x/\log^{1+\varepsilon} x)).$
\begin{theorem}\label{le:3}
	In the situation of {\rm Lemma \ref{le:1}} denote  $g(x)=\log(\psi(x)).$  Assume that 
	$$
	\displaystyle \int_{\psi^{-1}(e^2)}^\infty \frac{\log\log \psi (x)}{x^2}dx<\infty.
	$$
	Then  we may write for a general  annulus $A=\{r<|x|<R\}$ with $r<R\leq 1$
	$$
	2\pi\big(M_{\psi,I}(\Gamma)\big)^{-1}=o_{\psi,I_0}(1) +\frac{1}{2}\int_{g^{-1}(2\log(R^{-1}))}^{g^{-1}(2\log(r^{-1}))}\frac{g'(x)dx}{x}
	$$
	(where $\frac{d}{dx}[\log(\psi(x))]dx=d(\log(\psi(x)))$ if $\psi$ is not differentiable). The error term  satisfies $o_{\psi,I_0}(1)\to 0$ as $R\to 0$  for
	fixed $I$ and $\psi$.
\end{theorem}
\begin{proof} 
	Let us first consider the case  $r=e^{-n} <R=e^{-m}$ as in Lemma \ref{le:1}. We  obtain an upper bound for $1/M_{\psi,I}(\Gamma)$ by choosing $b_j =1/2j^2$  in the right hand side of Lemma \ref{le:1} since  $\psi^{-1}$ is increasing, and because  $\sum_{j=1}^{n-m} \frac{1}{2}j^{-2}\leq 1$. A lower bound for $1/M_{\psi,I}(\Gamma)$ is again obtained by replacing  $b_j=1$ in the left hand side of Lemma \ref{le:1}.  The difference between the upper and lower bound  is less than  (below we denote $\lfloor \log I_0\rfloor=a$)
	\begin{eqnarray*}\label{eq:10}
		S&:=& \sum_{j=m}^\infty \Big(\big(\psi^{-1}(I_0e^{2j}/2j^2)\big)^{-1} -\big(\psi^{-1}(I_0e^{2j+2})\big)^{-1} \Big)\\ &=&
		\sum_{j=m}^\infty \Big(\big(g^{-1}(\log I_0+2(j-\log j)-\log 2)\big)^{-1} -\big(g^{-1}(\log I_0 +2(j+1))\big)^{-1} \Big)\\
		\\&\leq&\sum_{j=2m+a}^\infty \Big(\big(g^{-1}((j-2\log (j-a))\big)^{-1} -\big(g^{-1}(j+2)\big)^{-1} \Big)\\
		&\leqq& \int_{2m+a}^\infty\Big(\big(g^{-1}(x-2\log (x-a))\big)^{-1}-\big(g^{-1}(x+3)\big)^{-1}\Big)dx.
	\end{eqnarray*}
	We  assume  momentarily that $\psi$ is differentiable. Denote $h=: -(1/g^{-1})'$. The last written integral equals
	\begin{eqnarray*}\label{eq:10a}
		\int_{2m+a}^\infty \bigg(\int_{x-2\log (x-a) }^{x+3} h(u)du\bigg)dx\; & \leq & \; \int_{2m+a-2\log(2m)}^\infty h(u) \big|\{x\;:\; x-2\log (x-a)\leq u\leq x+3\}\big|du\\
		&\leq & 3\int_{2m+a-2\log(2m)}^\infty h(u)(1+\log (u-a))du\\
		&= &3\int_{g^{-1}(2m-2\log(2m)+a)}^{\infty}\frac{1+\log (g(u)-a)}{u^2}\; =  E_0(m,I_0)\\
	\end{eqnarray*}
	where we applied Fubini and the substitution $u=g(x)$. By the assumptions this quantity is finite and obviously $ E_0(m,I_0)\to 0$ as $m\to\infty$ for any fixed $I_0.$ Finally, we may dispense with the assumption that $g$ is differentiable by simple approximation.
	
	The above argument shows that $2\pi\big(M_{\psi,I}(\Gamma)\big)^{-1}$ differs from the sum  $\sum_{j=m}^\infty\big(\psi^{-1}(I_0e^{2j+2})\big)^{-1}$ at most by the quantity $E_0$ defined above. As $x \mapsto \big(\psi^{-1}(I_0e^{2x})\big)^{-1}$ is decreasing, it follows that
	
\[
\int_{m}^{n+1} \big(\psi^{-1}(I_0e^{2x})\big)^{-1} dx \leq \sum_{j=m}^n\big(\psi^{-1}(I_0e^{2j+2})\big)^{-1} \leq \int_{m-1}^n\big(\psi^{-1}(I_0e^{2x})\big)^{-1} dx.
\]
By using the fact that $\psi^{-1}$ is increasing and bounded by 1 from above,  the difference between the sum and the sum where we replace $I_0$ by 1 decreases to zero uniformly in $n$   as $m$ tends to infinity. The same holds true for the difference of the RHS and LHS above. Hence, 
	together with the  change of variables  $x\mapsto\frac{1}{2}g(x)$ this yields 
	$$
	2\pi\big(M_{\psi,I}(\Gamma)\big)^{-1} =o_{\psi,I_0}(1) + \frac{1}{2}\int_{g^{-1}(2\log(R^{-1}))}^{g^{-1}(2\log(r^{-1}))}\frac{g'(x)dx}{x}
	$$
	where $o_{\psi,I_0}(1)\to 0$ as $R\to 0.$
\end{proof}

Let us then use Theorem \ref{le:3} to write down explicitly the modulus of the path family connecting inner and outer boundary of annuli $A=\{r<|x|<R\}$  in some simple cases: 

\begin{corollary}\label{cor:2} 
	(i)\quad Assume that $\psi(x)=e^{px^\alpha}$ with $\alpha\geq 1$. Then, if $\alpha=1$ we have
	\begin{equation}\label{eq:101} 
	M_{\psi,I}(\Gamma) = 2\pi\Big(o_{\psi,I_0}(1)+\frac{p}{2}\Big(\log\log (r^{-1})-\log\log(R^{-1})\Big)\Big)^{-1}
	\end{equation}
	and for $\alpha >1$
	\begin{equation}\label{eq:102}
	M_{\psi,I}(\Gamma) = 2\pi\Big(o_{\psi,I_0,\alpha}(1)+\big(\frac{p}{2}\big)^{1/\alpha}\frac{\alpha}{\alpha-1}\Big((\log^{(\alpha-1)/\alpha} (r^{-1})-\log^{(\alpha-1)/\alpha} (R^{-1})\Big)\Big)^{-1}.
	\end{equation}
	\noindent (ii)\quad Assume that $\psi(x)=e^{px/\log^\beta(x)}$ with $\beta\in (0,1]$ and $p>0.$ Then, if $\beta=1$ we have
	\begin{equation}\label{eq:101a}
	M_{\psi,I}(\Gamma) = 2\pi\Big(o_{\psi,I_0, p}(1)+\frac{p}{2}\Big(\log\log\log (r^{-1})-\log\log\log(R^{-1})\Big)\Big)^{-1}
	\end{equation}
	and for $\beta\in (0,1)$
	\begin{equation}\label{eq:102a}
	M_{\psi,I}(\Gamma) = 2\pi\Big(o_{\psi,I_0,\beta. p}(1)+\frac{p}{2(1-\beta)}\Big(\big(\log\log(1/r)\big)^{1-\beta}-\big(\log\log(1/R)\big)^{1-\beta}\Big)\Big)^{-1}.
	\end{equation}
\end{corollary}
Note that in these examples $\psi(1)>0$, but this has no effect on the results as one may easily check in these cases.
\subsection{Proof of Theorem \ref{LowerModCont}} 

We first consider boundary points, one of which we may assume to be $1$. We employ the Beurling type estimate (see e.g. \cite[Lemma 3.6]{GKT})
	\begin{equation}\label{eq:beurling}
	\diam(f^{-1}(E))\leq c\, \diam(E)^{1/2K} 
	\end{equation}
	for continua $E\subset f(\partial \D)$. For $z\in\partial\D$ we denote $|z-1|=d$ and assume that there are points $z\in \partial \D$ arbitrarily close to 1 with
	\begin{equation}\label{eq:beurling2}
	d':=|f(z)-f(1)|\leq d^\alpha=|1-z|^\alpha.
	\end{equation}
	where $\alpha >2K$. We then  seek for an upper bound for $\alpha.$ For that end, fix $\ell\geq 2$. Pick $z$ satisfying \eqref{eq:beurling2} 
	with $d<<1$ and note that we may pick disjoint continua $F_1,F_2\subset
	f(\partial \D\setminus L)$, where $L$ is the open shorter arc between 1 and $z$, and  with the following properties:
	\begin{eqnarray}\label{eq:beurling3}
	f(1)\in F_1,\quad f(z)\in F_2\qquad \textrm{and}\quad \diam(F_1),\; \diam(F_2)= \ell d' \leq \ell d^\alpha.
	\end{eqnarray}
	By  \eqref{eq:beurling} it follows that the preimages $I_j:=f^{-1}(F_j)$ are closed intervals of $\partial \D$ such that
	$|I_j|\leq c (\ell d^\alpha)^{1/2K}$, $j=1,2$.  And hence, as one of them contains the point $1$ and the other point $z$, we have $\dist (I_1,I_2)\gtrsim d$. Let $\Gamma$ be the family of curves that join $I_1$ to $I_2$. Then $f(\Gamma)$ is the family of curves joining $F_1$ to $F_2$. We obtain
	\begin{equation}\label{eq:beurling4}
	M_K(\Gamma)\geq M(f(\Gamma))\geq c_2\log(\ell+1)
	\end{equation}
	where $c_2>0$ is independent of $\ell$, and the last estimate is due to standard classical modulus estimate \cite[Lemma 7.38]{V}  since now $\ell \dist (F_1,F_2)\leq \diam(F_1)= \diam (F_2)$. To estimate the weighted modulus from above we denote by $\bar{\Gamma}$ the family of curves that join inner and outer boundary inside the annulus $A(1; c(\ell d^\alpha)^{1/2K}, d/2)$ and note that $M_K(\bar{\Gamma}) \geq M_K(\Gamma)$. 
	
	We get an upper estimate for $M_K(\bar{\Gamma})$ by  $M_{\psi,I_0}(\bar{\Gamma})$ with the choice  $\psi(x)=e^{px}-1$, and with  some value of $I_0$ that depends only on the map $f$. Thus by the  first part of Corollary \ref{cor:2} and \eqref{eq:beurling4} it follows for small enough $d$ that
	$$
	2\pi\Big(o_{\psi,I_0}(1)+\frac{p}{2}\Big(\log\log \big(c^{-1}(\ell d^\alpha)^{-1/2K}\big)-\log\log\big((d/2)^{-1}\big)\big)\Big)\Big)^{-1}\geq c_2\log(\ell+1).
	$$
Letting $d\to 0$ we obtain 
$$
	\alpha\;\leq \; 2K\exp\left(\frac{c}{p\log(\ell+1)}\right)
	$$
	which gives the desired estimate since $\ell$ may be taken arbitrarily large. Finally, we note that the case where one or both of the points lie in the interior can be handled exactly in the same manner.
	\begin{remark}\label{rem:other_psi}
	One could of course optimize over the choice of $\ell$ in the above proof, and obtain modulus of continuity for the inverse map on the boundary that decays slower than any power $t^{\alpha}$ with $\alpha >2K$.
	
		Moreover, given a more relaxed integrability condition for the distortion $K_f$ we can obtain a lower bound for the modulus of continuity by combining Theorem \ref{le:3} 
		with the proof of 
		Theorem \ref{LowerModCont}. This works 
		as long as the weighted modulus $M_K(\bar{\Gamma})$
		 converges to zero, and  the range in which this happens can be seen from Corollary \ref{cor:1}, and thus  it includes  the so called sub-exponentially integrable mappings.  As  an example we state the following result (leaving the details of the proof to the reader):
\begin{corollary}\label{cor:subexp} Let $f$ be a planar mapping  of finite distortion such that $\displaystyle \exp\Big(\frac{pK_f}{1+\log K_f}\Big)\in L^1_{loc} $ for some $p$. Assume also that $f$ is $K$-quasiconformal in the unit disk. Then for any $c>1$ we have
	\begin{equation}\label{LowerModContyhtA}
	 |f(z)-f(w)| \gtrsim \exp\Big( -\big(\log(1/|z-w|)\big)^c\Big)\qquad \textrm{\rm for any } z,w\in \partial \D.
	\end{equation} 
\end{corollary}	
\noindent We observe  that the continuity of the inverse map on the boundary has again improved drastically surpassing the general estimate for  mappings of exponentially integrable distortion, see \cite[Theorem B and Example B]{HK}, which is basically obtained by taking $c=2$ in \eqref{LowerModContyhtA}. In this  connection it is useful to note that in general the modulus of continuity for inverses in the class $ \exp(\frac{pK_f}{1+\log K_f})\in L^1_{loc}$ is still  weaker as the example
$$
f(z)=\frac{z}{|z|}\exp\left(-\frac{1}{p+\varepsilon}\big(\log(1/|z|)\big)^2\log\log(1/|z|)\right)
$$
demonstrates - for this and other results on modulus of continuity of the inverse in the subexponential class we refer to \cite{CH}.	
	\end{remark}
	
	\subsection{Proof of Corollary \ref{co:rotation}}
	We will finish this section  by  using Theorem \ref{LowerModCont} to find an upper bound for the pointwise rotation proving Corollary  \ref{co:rotation}. Thus, let $f$ be a planar mapping with $p$-exponentially integrable distortion that is $K$-quasiconformal in the upper half-plane and is normalized by $f(0)=0$ and $f(1)=1$. One of our main tools will again be the classical modulus bound 
	\begin{equation}\label{KvasimodulusIE}
	M(f(\Gamma)) \leq K M(\Gamma)
	\end{equation}
	that holds for any path family $\Gamma \subset \mathbb{H}$ and any $K$-quasiconformal mapping $f:\mathbb{H} \to \mathbb{C}$. 
	\medskip
	
	Fix $r>0$ and denote $E=[r,1]$ and $F=(-\infty,0]$. Set also $\Gamma$ to be the family of all paths connecting $E$ and $F$ in the upper half-plane. Our aim is to bound from above the winding of $f([r,1])$ around the origin using the modulus inequality \eqref{KvasimodulusIE}, which means that we must find good estimates for the moduli.  On the domain side this is simple as we can use classical estimates, see \cite[Theorem 7.26]{V}, to get
	\begin{equation}\label{Modulus bound from above proof rotation}
	M(\Gamma) \leq c \log\left( \frac{1}{r} \right)
	\end{equation}
	where $c$ is an universal constant and $r$ is small.
	
	Thus we can concentrate on estimating $M(f(\Gamma))$. First we write the modulus in polar coordinates by 
	\begin{equation*}
	M(f(\Gamma)) = \inf_{\rho\hspace{0.1cm}admissible}\int_{0}^{2\pi}\int_{0}^{\infty}\rho^{2}(t,\theta)t\hspace{0.1cm}dtd\theta
	\end{equation*}
	and establish a lower bound for 
	\begin{equation*}
	\int_{0}^{\infty}\rho^{2}(t,\theta)tdt
	\end{equation*}
	that is independent of the direction $\theta$ and admissible function $\rho$. Assume that the image $f([r,1])$ winds $n(r)+1$ times around the origin and denote by $L_\theta$ the half line starting from the origin to a given direction $\theta$. Since $f$ is a homeomorphism it is clear that we can find $n(r)$ separate line segments  $\left(x_{j}e^{i\theta},y_{j}e^{i\theta}\right)\subset L_\theta$ such that each $\left(x_{j}e^{i\theta},y_{j}e^{i\theta}\right)$ belongs to $f(\Gamma)$. Furthermore, these positive real numbers $x_j$,$y_j$  satisfy 
	$$\aligned
	0<r_f\leq x_1<y_1<...<x_{n(r)}<y_{n(r)}\leq c_f,
	\endaligned$$ where $c_f=\sup_{z\in E}|f(z)|$ and $r_f=\min_{z\in E}|f(z)|$ and we note that neither $c_f$ nor $r_f$ depend on $\theta$. Next we estimate
	$$\aligned
	\int_{0}^{\infty}\rho^{2}(t,\theta)tdt\geq\sum_{j=1}^{n(r)}\int_{x_j}^{y_j}\rho^{2}(t,\theta)tdt
	\endaligned$$
	and for every individual integral we can use H\"older, since $\rho$ is admissible and segments  $\left(x_{j}e^{i\theta},y_{j}e^{i\theta}\right)$ belongs to $f(\Gamma)$, to estimate
	$$\aligned
	\int_{x_j}^{y_j}\rho^{2}(t,\theta)tdt\geq\left(\int_{x_j}^{y_j}\rho(t,\theta)dt\right)^{2}\left(\int_{x_j}^{y_j}\frac{1}{t}dt\right)^{-1}\geq\frac{1}{\log\left(\frac{y_j}{x_j}\right)}.
	\endaligned$$
	Therefore,
	$$\aligned
	\int_{0}^{\infty}\rho^{2}(t,\theta)tdt\geq\sum_{j=1}^{n(r)}\frac{1}{\log\left(\frac{y_j}{x_j}\right)}
	\endaligned$$
	and by the choice of numbers $x_j$ and $y_j$ it is clear that
	$$\aligned
	\sum_{j=1}^{n(r)}\frac{1}{\log\left(\frac{y_j}{x_j}\right)}\geq\sum_{j=1}^{n(r)-1}\frac{1}{\log\left(\frac{x_{j+1}}{x_j}\right)}+\frac{1}{\log\left(\frac{c_f}{x_{n(r)}}\right)}.
	\endaligned$$
	Therefore we can use the arithmetic-harmonic-means inequality to estimate
	$$\aligned
	\sum_{j=1}^{n(r)}\frac{1}{\log\left(\frac{y_j}{x_j}\right)}\geq\frac{n^{2}(r)}{\log\left(\frac{c_f}{x_1}\right)}\geq\frac{n^{2}(r)}{\log\left(\frac{c_f}{r_f}\right)}.
	\endaligned$$
	The upper bound $c_f$ is independent of $r$ and insignificant as $r \to 0$. Thus we can combine our previous estimates and use Theorem \ref{LowerModCont}, when $r$ is small, to obtain
	\begin{equation*}
	\int_{0}^{\infty}\rho^{2}(t,\theta)tdt \geq \frac{n^{2}(r)}{cK\log \left( \frac{1}{r} \right)},
	\end{equation*} 
	where $c$ and $C$ are universal constants. Since this bound holds for an arbitrary direction $\theta$ and any admissible $\rho$ we get  
	\begin{equation}\label{Modulus bound from below Rotation proof}
	M(f(\Gamma)) \geq \frac{c n^{2}(r)}{K\log \left( \frac{1}{r} \right)},
	\end{equation}
	where $c$ and $C$ are again universal constants. Combining the bounds \eqref{Modulus bound from above proof rotation} and \eqref{Modulus bound from below Rotation proof} with the modulus inequality we get 
	\begin{equation*}
	\frac{c n^{2}(r)}{K\log \left( \frac{1}{r} \right)} \leq cK \log \left( \frac{1}{r} \right),
	\end{equation*}
	which leads to the desired bound for the winding number $n(r)$ and thus finishes the proof of Corollary \ref{co:rotation}.

Finally we note that  also in more general situations, after establishing a lower bound for the modulus of continuity on the boundary $\partial \D$, see Remark \ref{rem:other_psi}, we can use the above proof word by word to find an upper bound for the pointwise rotation.

\section{Regularity at cusps}\label{se:regularity}

In the present section we study regularity of  conformal map $\D\to\Omega$ in the situation where the  domain $\Omega$ is rather smooth apart from outward cusps. Besides intrinsic interest, this relates in a nice way to the general topic of the paper since one expects that extremal situations to regularity in the general setup of the present paper  arise from a single cusp. Namely, in analogous situations  for quasiconformal  extensions the extremal maps correspond to domains with outward pointing angle, see  \cite{PS,Pr19}.

We start by studying  the regularity of a conformal map in a neighbourhood of a boundary point $w_0\in\partial \D$ of the unit disk where the derivative  essentially increases upon approaching $w_0$. The following lemma --  basically a variant of Hardy-Littlewood-P\'olya's  re-arrangement  inequality -- estimates  local regularity of $f'$ in such a situation as soon as a bound for  the modulus of the continuity of the map in a neighbourhood of $w_0$ is known.  
\begin{lemma}\label{estiLem}
Assume that $f:\D \to \Omega$ is a conformal map, where  $\Omega \subset \D$ is a bounded Jordan domain, and let us denote also by $f$ the homeomorphic extension $\overline{\D}\to\overline{\Omega}$. Assume that $f$ satisfies the following  modulus of continuity estimate on the boundary arc $A:=
\{ e^{it}\; : \; t\in (-\varepsilon_0,\varepsilon_0)\} \subset\partial\D$: 
\begin{equation}\label{eq:jep1}
|f(z_1)-f(z_2)| \leq \psi (|z_1-z_2|), \qquad \quad z_1, z_2\in  A,
\end{equation}
where  $\psi:[0,\infty)\to [0,\infty)$ is strictly increasing,  satisfies $\psi(0)=0$, and   is differentiable on $(0,\infty)$.
Moreover, assume  that $f'$ extends continuously to the punctured arc $A\setminus\{1\}$ and there is a finite constant $C$ such that 
\begin{equation}\label{eq:jep2}
|f'(e^{it})|\leq C |f'(e^{it'})|\;\;\textrm{and}\;\; |f'(e^{-it})|\leq C |f'(e^{-it'})|
\end{equation}
together with
\begin{equation}\label{eq:jep3A}
|\arg f'(e^{it})-\arg f'(e^{it'})|,\;  |\arg f'(e^{-it})-\arg f'(e^{-it'})| \;\leq \pi/2
\end{equation}
for all $0<t'<t<\varepsilon_0$.
Then for every convex and increasing function $\Phi:(0,\infty)\to (0,\infty)$  with $\Phi (\lambda x)\leq C_\lambda\Phi(x)$ for all $\lambda>1$ and $x>0$ we have
$$
\int_{\D\cap B(1,\varepsilon/2)}\Phi \big(|f'(z)|\big)dA(z)\lesssim \int_{\D\cap B(1,\varepsilon/2)}\Phi \big(\psi'(|z|)\big)dA(z).
$$
% A(dz) notation for area integral? consistency with the rest
\end{lemma}
\begin{proof} We may replace $\D$ by the upper half plane $\mathbb{H}$  for notational simplicity, and scale the situation so that the arc $A$ corresponds to the interval $(-2,2)\subset\R=\partial \mathbb{H}.$ By symmetry, it suffices to consider the integral over the right half of a neighbourhood $R:=(-2,2)\times (0,1)\subset\mathbb{H}$. We divide this in the standard way (up to measure zero) to Whitney cubes $Q_{n,k}$, where $n\geq 0$ and  $1\leq k\leq 2^n$, so that $Q_{n,k}$ has the left lower corner at $(k-1)2^{-n}+2^{-n}i$. Denote by $I_{n,k}$ the vertical projection of $Q_{n,k}$ to the real axis. By an easy  modulus estimate we have for any $z\in Q_{n,k}$
$$
|f'(z)| \lesssim\mathrm{diam}f(I_{n,k})|I_{n,k}|^{-1}.
$$
Namely, if $B$ stands for the hyperbolic ball of radius 1 in $\mathbb{H}$ centered at $z$, its image $f(B)$ is a hyperbolic ball in $f(\mathbb{H})$, and by basic hyperbolic geometry its diameter and distance from the boundary are comparable to $|f'(z)| |I_{n,k}|$. Letting $\Gamma$ stand for the curves in $\mathbb{H}$ connecting $B$ and $I_{n,k}$, we have  $M(\Gamma)\approx 1.$ By conformal invariance of modulus also $M(f(\Gamma))\approx 1$, which implies the claim by a standard comparison of the modulus to a modulus of an annulus, since $f(B)$ is a hyperbolic ball  of radius one in $f(\mathbb{H})$, and  by basic hyperbolic geometry both its diameter and its distance from the boundary $\partial (f(\mathbb{H}))$ are comparable to $|f'(z)| |I_{n,k}|$.

Let us define for $x\in (0,1]$ the  function $u(x)=(2C)^{-1}\sup_{t\in [x,1]}|f'(x)|.$ Then $u$ is increasing and it  satisfies $u(x)\approx |f'(x)|$ on $(0,1]$ by the first condition in \eqref{eq:jep2}, whence for each $n,k$ it holds that
$$
|f'(z)| \lesssim 2^{n}\int_{I_{n,k}}u\quad \textrm{for}\quad z\in Q_{n,k}.
$$
In turn, by condition \eqref{eq:jep2} we see that for all $0\leq x<y\leq 1$
\begin{equation}\label{eq:jep3}
\int_x^y u(t)dt\lesssim \psi(y-x).
\end{equation}
We may compute for each $n\geq 0$
\begin{eqnarray*}
\int_{\cup_{k=1}^{2^n}Q_{n,k}}\Phi\big(|f'(z)|\big) dA(z)&\lesssim&2^{-2n}\sum_{k=1}^{2^n}\Phi\Big(2^{n}\int_{I_{n,k}}u(x)dx\Big)
\;\leq\;2^{-n}\int_0^1\Phi\big(u(x)\big)dx,
\end{eqnarray*}
where we applied Jensen's inequality. By summing over $n$ we finally obtain
$$
\int_R\Phi\big(|f'(z)|\big) dA(z)\leq \int_0^1\Phi\big(u(x)\big)dx\leq \int_0^1\Phi\big(\psi'(x))\big)dx,
$$
where the last written inequality is a consequence of  \eqref{eq:jep3} and a continuous version of Hardy-Littlewood-P\'olya inequality, see e.g. \cite[Lemma 11]{HLP}.
\end{proof}

The above lemma is tailored towards   studying  regularity of a conformal map at a boundary point that corresponds to an  outward cusp. To make things precise, we  define the notion of  \emph{outward $C^{1+\varepsilon}$-cusp}: Let $\Omega\subset\C$ be a simply connected domain and let $z_0\in\partial\Omega$. We say that $\Omega$ has an outward $C^{1+\varepsilon}$-cusp at $z_0$ if there is an open square $Q$ centered at $z_0$ such that, modulo scaling and isometry, we may write
$$
Q\cap\Omega=\{z\in\C \; :\;  0<{\rm Re \,}z <1\quad\textrm{and}\quad \psi_1(x)<{\rm Im \,}z <\psi_2(x)\},
$$
where $\psi_1,\psi_2\in C^{1+\varepsilon}([0,1])$ satisfy $\psi_1<\psi_2$ on $(0,1]$ and $\psi'_1(0^+)=\psi'_2(0^+)$=0. If the last condition is replaced by
$\psi'_1(0^+)<\psi_2'(0^+)$, we speak of \emph{regular outward $C^{1+\varepsilon}$-corner.}

In the next result we exclude corners since the regularity of the conformal map  around a point corresponding to a outward $C^{1+\varepsilon}$-corner at the image side is easy to control, as it is up to a constant  the same as that of the map $(1-z)^{\alpha/\pi}$, where $\alpha$ is the inward angle. This is easily seen by `flattening the corner' by a suitable power map.

\begin{theorem}\label{lConv}
Assume that $f:\D\to \Omega$ is a conformal map, where $\Omega$ is a Jordan domain with $C^{1+\varepsilon}$-regular boundary apart from finitely many points that are tips of outward $C^{1+\varepsilon}$-cusps. Assume also that $f$   satisfies the modulus of continuity 
\[
|f(z_1)-f(z_2)| \leq \psi(|z_1-z_2|)\qquad \textrm{ for all } \quad z_1,z_2 \in \partial \D,
\]
where $\psi:[0,\infty)\to [0,\infty)$ satisfies $\psi(0)=0$,  strictly increasing,  and is differentiable on $(0,\infty)$. Then for every convex and increasing $\Phi:(0,\infty)\to (0,\infty)$  with $\Phi (\lambda x)\leq C_\lambda\Phi(x)$ for all $\lambda>1$ and $x>0$ we have
$$
\int_{\D}\Phi \big(|f'(z)|\big) dA(z)< \infty\quad \textrm{whenever}\quad  \int_{\D}\Phi \big(\psi'(|z|)\big) dA(z) <\infty.
$$
\end{theorem}
\begin{proof}
This follows from Lemma \ref{estiLem} as soon as we verify that the conformal map satisfies first condition in \eqref{eq:jep2} at a cusp point. This will be verified in Lemma \ref{le:lineEstim} below.
\end{proof}

We are ready for:
\begin{proof}[Proof of Theorem \ref{cor:exp_p_cusps}]
The result is an immediate consequence of Theorem \ref{lConv} and Zapadinskaya's result  \cite[Corollary 1]{Zap} which states that $f$ satisfies the 
modulus of continuity estimate
\[
|f(z_1)-f(z_2)| \leq C \frac{1}{\log^p \left( \frac{1}{|z_1-z_2|} \right)} \qquad \textrm{ for all } \quad z_1,z_2 \in \partial \D.
\]
\end{proof}
% We will use harmonic measure to prove the claim. Choose origin as starting point and note that each $I_j$ has the same harmonic measure. The conformal mappings preserve the harmonic measure, so on the image side $f(I_j)$ have the same harmonic measure when the starting point is $f(0)$.

Finally, we establish the needed auxiliary result referred to above.

\begin{lemma}\label{le:lineEstim}
Let $f:\D\to\Omega$ be conformal map and assume  that $f(1)$ is the tip of an outward $C^{1+\varepsilon}$-cusp of $\Omega$. Then $f$ satisfies the first condition in \eqref{eq:jep2}.

\end{lemma}
\begin{proof} 
We need to show that $|f'(e^{it})| \leq C |f'(e^{it'})|$ for every $0 < t' < t < \varepsilon_0$. Let $\varepsilon_0$ be small enough so that the curve $\{f(e^{is}) : s \in [0,2\varepsilon_0]\}$ is $C^{1+\varepsilon}$. Then we can find a simply connected $C^{1+\varepsilon}$-domain $\Omega'$ such that $\Omega \subset \Omega'$ and $T := \{f(e^{is}) : s \in [0,\varepsilon_0]\} \subset \partial \Omega'$.

The Riemann mapping theorem implies that there exists a conformal mapping $g: \Omega' \to \mathbb{H}$ that has $g(f(1))=0$ and $g(f(e^{is})) \neq \infty$ for $s \in [0,\varepsilon_0]$. The derivative of this mapping is $C^{1+\varepsilon}$ on the boundary outside of the pre-image of infinity \cite[Theorem 3.6]{Pom}, so in particular the restriction $\tilde{g} : \Omega \to g(\Omega)$ extends to a $C^{1+\varepsilon}$-function on the closure and has $\frac{1}{C_0} \leq \tilde{g}'(w) \leq C_0$ for all $w \in \overline{\Omega}$.

Denote $h = \tilde{g} \circ f$. Since
 $|f'(e^{it})|=|(\tilde{g}^{-1})'(h(e^{it})) h'(e^{it})| \leq C_0 |h'(e^{it})|$ and similarly $|f'(e^{it'})|=|(\tilde{g}^{-1})'(h(e^{it'})) h'(e^{it'})| \geq \frac{1}{C_0} |h'(e^{it'})|$, it suffices to show that $h$ satisfies the first condition in \eqref{eq:jep2}. Since $g$ is $C^{1+\varepsilon}$, then also $g(\Omega)$ has an outward $C^{1+\varepsilon}$-cusp at $g(f(1))$.

Moreover, by considering a M\"obius mapping $\tau$ taking the unit disc $\D$ to upper half-plane $\mathbb{H}$ with $\tau(1)=0$, we observe that this mapping has again bounded derivatives in the pre-image of $T$, so it suffices to show that for $\varphi = \tau^{-1} \circ h : \mathbb{H} \to g(\Omega)$ that $\varphi'(x_2) \leq C \varphi'(x_1)$  for $0 < x_1 < x_2 < \delta$ for some $\delta > 0$.

Assume first that $\varphi(x_1)<100\varphi(x_2)$. Using Schwarz reflection principle, we can extend $\varphi$ analytically to obtain a conformal mapping from $A_1 := \C \setminus \{z \in \C : \Im z = 0, \Re z \notin (0,1)\}$.

Within a hyperbolic disk of radius $r$ centered at $z_0$ in any domain $A$, there is a constant $C = C_r$ such that for any $z$ in the hyperbolic disk and any conformal $f$ we have
\begin{equation}\label{eq:reg1}
\frac{1}{C} |f'(z_0)| \leq |f'(z)| \leq C |f'(z_0)|.
\end{equation}
The claim follows then by the fact that the segment between $\varphi(x_1)$ and $\varphi(x_2)$ can be covered by finitely many hyperbolic disks of constant radius, where the number of hyperbolic disks does not depend on the choice of $x_1,x_2$.

We may thus assume that $\varphi(x_1)\geq 100\varphi(x_2)$.
Consider intervals $I$ and $I_0$ that contain $\varphi(x_1)$ and $\varphi(x_2)$ respectively. By choosing the segments $I$ and $I_0$ to have the same hyperbolic length in the reflected domain $A_1$, we observe in view of \eqref{eq:reg1} that it suffices to show that the harmonic measure $\omega_{\Omega,z_0}$ satisfies
\[
\frac{\omega_{\Omega,z_0}(I_0)}{|I_0|} \leq c \frac{\omega_{\Omega,z_0}(I)}{|I|},
\]
where  $z_0$ is a fixed point away from the cusp and  $c$ is some constant that is independent of the chosen intervals $I,I_0$ near the tip of the cusp.
The above fact follows from Lemma \ref{le:harmonic} below.

\end{proof}

\begin{lemma}\label{le:harmonic}
Assume that $\Omega$ is a bounded simply connected  domain containing an outward  cusp of the form
$$
\{ 0\leq y\leq g(x), \quad 0< x <1\},
$$
where the function $g:[0,1]\to \R$ belongs to $C^{1+\varepsilon }([0,1])$, is positive on $(0,1]$ and satisfies $g(0)=0$. Fix $z_0\in \Omega$ with ${\rm Re\,}z_0>3/4$  and denote by $\omega_{\Omega, z_0}$ the harmonic measure on $\partial\Omega$ with respect to $z_0$. Then for small enough $a\in (0,1/2)$ the following holds true: Let $0<b<a$. Denote $I=[a-g(a)/2,a+g(a)/2]$ and let $I'=[b-g(b),b],$ where $b<a-g(a)/2.$ Then there is $a_0<1$ and  a finite constant $c$, independent of the intervals $I,I'$, so that for $a\leq a_0$  the harmonic measures of the intervals $I$ and $I'$ satisfy the comparison
$$
\omega_{\Omega, z_0}(I')\leq c\frac{|I'|}{|I|}\omega_{\Omega, z_0}(I).
$$
\end{lemma}
\begin{proof}
Let $L$ be the vertical line $(a, a+ig(a))$. For $z\in\Omega$ denote $h(z):=\omega_{\Omega, z}(I)$ and $h_1(z):=\omega_{\Omega, z}(I')$. Now $L$ cuts $\Omega$ in two subdomains,  and denote by $\Omega'$ the part 'to the right' of $L$. By considering the harmonic function $c\frac{|I'|}{|I|}h-h_1$ in $\Omega'$ and comparing boundary values we see that it suffices to prove that 
$$
h_1(z)\leq c\frac{|I'|}{|I|}h(z)\quad\textrm{for}\quad z\in L.
$$
We claim that it is actually enough to show this for $z$ lying on the lower half $L_0:=(a, a+ig(a)/2)$ of the segment $L$:
\begin{equation}\label{eq:enough}
h_1(z)\leq c\frac{|I'|}{|I|}h(z)\quad\textrm{for}\quad z\in L_0.
\end{equation}
Namely, since $h,h_1$ are positive harmonic functions vanishing on the upper half on $L$, the desired inequality for the upper half of $L$ follows from the boundary Harnack principle (see Lemma \ref{le:bharnack}   below) as soon as we know it for the middle point $a+ig(a)/2$.

In order to show \eqref{eq:enough}, we note first that $h(z)\geq c_0$ for any $z\in (a, a+ig(a)/2)$, where $c_0$ is independent of the intervals. This is seen e.g. by considering a suitable isosceles triangle $T\subset\Omega$ whose  base is a subinterval of $I$ and the ratio of the base to the height is decided depending on the size of $g'(0)$. We note that $h(z)\geq \omega_{T,z}( I)\geq c_0$, where $c_0$ is independent of $I$ by scaling. In order to estimate $h_1$ we make an essentially maximal extension of $\Omega$, and replace it by the domain
$\Omega_0=\C\setminus (-\infty,a],$ and we need to show that $\omega_{\Omega_0,z}(I')\leq c|I'|/|I|$ for $z\in L_0$. After translation and scaling we may assume that $|I|=1$ and use the map $z\mapsto\sqrt{z}$ to replace the domain by the right half plane $\C_+:=\{ x>0\}$. It remans to check that
$$
\omega_{{\mathbb C}_+,z}(I_1)\leq C |I_1|
$$
for segments $I_1\subset [i,i\infty)$ (i.e. the images of intervals $I'$) with length less than 1, and for $z$ lying on the segment $(0,(1+i)/2)$. This is now easy by considering the angle with the vertex at $z$ and with base $I_1.$
\end{proof}

For the readers convenience we give a short proof  of  an instance of the well-known boundary Harnack principle, see  \cite{Ke}, where also the quantitative dependence is in the form that suits our needs.

\begin{lemma}\label{le:bharnack}  Let $\Omega$ be a bounded Jordan domain and  
let $A$ be an open boundary arc of $\partial\Omega$. Assume that the ball $B$ has center in $A$ and satisfies $rB\cap \partial\Omega\subset A$ with $r>1$. Then for any  positive harmonic functions $h_1,h_2$ on $\Omega$, both with vanishing boundary values on $A$, it holds that 
$$
\frac{h_1(z)}{h_2(z)}\leq  c\frac{h_1(w)}{h_2(w)} \quad\textrm{for all}\quad z,w \in B\cap \Omega,
$$
\end{lemma}
\begin{proof} Let $\Gamma$ be the path family consisting of curves in $\Omega$ that join $B\cap\overline{\Omega}$ to $\partial\Omega\setminus A$. By the assumption and simple comparison to the moduli of an annulus we see that $\Mod(\Gamma)\leq 2\pi(\log r)^{-1}:=c(r)$. We then push the whole  situation  to the unit disc by a conformal mapping $\phi:\Omega\to\D$, and denote the continuous extension to the boundary also by $\phi$. In addition, write  $\phi(B\cap\overline{\Omega})=:B'$ and $\phi(A)=:A'\subset\partial\D$ together with $\Gamma'=\phi(\Gamma)$. 
Since $\Mod(\Gamma')=\Mod (\Gamma)$,  applying the modulus estimate \cite[Lemma 7.38]{V}  for the connected subsets of the plane we see
that the relative distance of $B'$ and $\partial \D\setminus A'$ is bounded from below, i.e.
$$
\frac{\dist(B',\partial \D\setminus A')}{\diam(\partial \D\setminus A')\wedge\diam(B')}\geq c'(r)>0. 
$$
In the above calculation we note that by \cite[Lemma 5.20]{V} the modulus $M(\Gamma')$ changes only by a fixed constant when we momentarily consider path family that consists of all paths in the whole plane connecting the sets $B'$ and $A'\subset\partial\D$.
 
 Hence one of the following  two alternatives must be true:
\begin{equation}\label{eq:reg2}
\diam (\partial \D\setminus A')\leq c''(r)\dist(B',\partial \D\setminus A')\qquad \textrm\qquad \diam (B')\leq c''(r)\dist(B',\partial \D\setminus A').
\end{equation}

It turns out that the geometric information \eqref{eq:reg2} is exactly what we need for our purposes. By the Poisson kernel representation we have for any $z\in\D$ (and $j=1,2$)
\begin{equation}\label{eq:reg3}
h_j(\phi^{-1}(z))=\int_{\partial\D\setminus A'}\frac{1-|z|^2}{|z-\alpha|^2}\mu_j(d\alpha)
\end{equation}
where $\mu_j$ is a positive Borel measure supported on $\partial\D\setminus A'$. If the first alternative in \eqref{eq:reg2} holds, we may fix a point $\alpha_0\in\partial \D\setminus A'$ and deduce that $|z-\alpha|\approx_{r}|z-\alpha_0|$ for all $\alpha\in\partial\D\setminus A'$ and $z\in B',$ whence the Poisson representation verifies that
$$
h_j(z)\approx_{r}\mu_j(\partial\D\setminus A')(1-|\phi(z)|^2)|\phi(z)-\alpha_0|^{-2},\qquad z\in B\cap \Omega
$$
which obviously  proves the claim in this situation.  If the second alternative in \eqref{eq:reg2} holds, we fix $z_0\in B'$  obtaining $|z-\alpha|\approx_{r}|z_0-\alpha|$ for all $\alpha\in\partial\D\setminus A',$ and $z\in B'$  and
$$
h_j(z)\approx_{r}(1-|\phi(z)|^2)\int_{\partial\D\setminus A'}\frac{1}{|z_0-\alpha|^2}\mu_j(d\alpha),
$$
which yields the claim in this case.
\end{proof}

\end{document}